\documentclass[nosumlimits,twoside]{amsart}
\usepackage{amsfonts, amsmath, amssymb}
\usepackage[bookmarksnumbered,plainpages,hypertex]{hyperref}
\usepackage{graphicx}
\usepackage{float}
\usepackage{srcltx}
\usepackage[all]{xy}
\usepackage{version}
\usepackage[usenames]{color}

\newtheorem{theorem}{\sc Theorem}[section]
\newtheorem{proposition}[theorem]{\sc Proposition}
\newtheorem{notation}[theorem]{\sc Notation}

\newtheorem{lemma}[theorem]{\sc Lemma}
\newtheorem{corollary}[theorem]{\sc Corollary}
\theoremstyle{definition}
\newtheorem{definition}[theorem]{\sc Definition}

\newtheorem{example}[theorem]{\sc Example}

\theoremstyle{remark}
\newtheorem{remark}[theorem]{\sc Remark}

\newtheorem{claim}[theorem]{}

\newtheorem{question}[theorem]{\sc Question}
\allowdisplaybreaks
\IfFileExists{tcilatex.tex}{\input{tcilatex}}{}

\excludeversion{proof?}
\newcommand{\id}{{\sf id}}

\newcommand{\Aa}{\mathcal{A}}

\def\Bialg{{\sf Bialg}}

\def\RLie{{\sf {Lie}_{p}}}

\begin{document}
\title[Restricted Lie algebras via monadic decomposition]{Restricted Lie
algebras via monadic decomposition}
\thanks{This note was written while A. Ardizzoni was member of GNSAGA and
partially supported by the research grant ``Progetti di Eccellenza
2011/2012'' from the ``Fondazione Cassa di Risparmio di Padova e Rovigo''.
He thanks the members of the department of Mathematics of both Vrije
Universiteit Brussel and Universit\'{e} Libre de Bruxelles for their warm
hospitality and support during his stay in Brussels in August 2013, when the
work on this paper was initiated. The second named author is a Marie Curie
fellow of the Istituto Nazionale di Alta Matematica.}

\begin{abstract}
We give a description of the category of restricted Lie algebras over a
field $\Bbbk $ of prime characteristic by means of monadic decomposition of
the functor that computes the $\Bbbk $-vector space of primitive elements of
a $\Bbbk $-bialgebra.
\end{abstract}

\author{Alessandro Ardizzoni}
\address{%
\parbox[b]{\linewidth}{University of Turin, Department of Mathematics ``G. Peano'', via
Carlo Alberto 10, I-10123 Torino, Italy}}
\email{alessandro.ardizzoni@unito.it}
\urladdr{\url{sites.google.com/site/aleardizzonihome}}
\author{Isar Goyvaerts}
\address{%
\parbox[b]{\linewidth}{University of Turin, Department of Mathematics ``G. Peano'', via
Carlo Alberto 10, I-10123 Torino, Italy}}
\email{isarrobert.goyvaerts@unito.it}
\author{Claudia Menini}
\address{%
\parbox[b]{\linewidth}{University of Ferrara, Department of Mathematics and Computer Science, Via Machiavelli
35, Ferrara, I-44121, Italy}}
\email{men@unife.it}
\urladdr{\url{sites.google.com/a/unife.it/claudia-menini}}
\maketitle
\tableofcontents

\section*{Introduction}

Let $\Bbbk $ be a field and let $\mathfrak{M}$ denote the category of $\Bbbk$%
-vector spaces. Denoting $\mathsf{Alg}(\mathfrak{M})$ the category of
unital, associative $\Bbbk$-algebras, there is the obvious forgetful functor
$\Omega:\mathsf{Alg}(\mathfrak{M})\to \mathfrak{M}$, which has a left
adjoint $T$. The composition $\Omega T$ defines a monad on $\mathfrak{M}$
and the comparison functor ${\Omega}_{1}$ from $\mathsf{Alg}(\mathfrak{M})$
to ${_{\Omega T}\mathfrak{M}}$ -the Eilenberg-Moore category associated to
the monad $\Omega T$- can be shown to have a left adjoint $T_1$ such that
the adjunction $(T_1,\Omega _1)$ becomes an equivalence of categories, i.e. $%
\Omega$ is a monadic functor. \newline
It is well-known that for any $V\in\mathfrak{M} $, $TV$ can be given
moreover a $\Bbbk$-bialgebra structure, thus inducing a functor $\widetilde{T%
}:\mathfrak{M}\to \mathsf{Bialg}(\mathfrak{M})$. Now, a right adjoint for $%
\widetilde{T}$ is provided by the functor $P$ that computes the space of
primitive elements of any bialgebra. This adjunction furnishes $\mathfrak{M}$
with a (different) monad $P\widetilde{T}$. This time, $P$ fails to be
monadic, alas. Indeed, $P_1$ -the comparison functor associated to the monad
$P\widetilde{T}$- still allows for a left adjoint ${\widetilde{T}}_1$, but
the adjunction $({\widetilde{T}}_1,P_1)$ is not an equivalence anymore. Yet,
something can be done with it. Using the notation ${\mathfrak{M}}_2$ for the
Eilenberg-Moore category of the monad $P_1 {\widetilde{T}}_1$ on ${_{ P
\widetilde{T}}\mathfrak{M}}$ (the Eilenberg-Moore category of the monad $P%
\widetilde{T})$, it was proven in \cite{AGM-MonadicLie1} that there exists a
functor
\begin{equation*}
{\widetilde{T}}_2: {\mathfrak{M}}_2\to \mathsf{Bialg}(\mathfrak{M})
\end{equation*}
that allows a right adjoint $P_2$ and which is moreover full and faithful.
This means that the functor $P$ has so-called ``monadic decomposition of
length at most 2''. \newline
In case the characteristic of the ground field $\Bbbk $ is zero, the above
result was further refined in \cite{AM-MonadSym}. Indeed, amongst other
things in the cited article, it is proven that the category ${\mathfrak{M}}%
_2 $ is equivalent with $\mathsf{Lie}(\mathfrak{M})$, the category of $\Bbbk$%
-Lie algebras. This theorem is actually obtained as a consequence of a
more general statement (\cite[Theorem 7.2]{AM-MonadSym}) that is proven for
Lie algebras in abelian symmetric monoidal categories that satisfy the so-called Milnor-Moore condition. It is then verified, using a result of Kharchenko (\cite[Lemma 6.2]{Khar}), that the category of vector spaces over a field of characteristic zero satisfies this condition.
In concreto, there exists a functor $\Gamma$ such that $%
(P_2 \widetilde{U}, \Gamma)$ gives the afore-mentioned equivalence, $%
\widetilde{U}$ being the functor that computes the universal enveloping
bialgebra of any Lie algebra in characteristic zero. \newline
\newline
In case the characteristic of $\Bbbk $ is a prime number $p$, things appear
to be slightly different. Of course, one can still work with the ordinary
definition of Lie algebra and consider its universal enveloping algebra
(which is still a $\Bbbk $-bialgebra, also in finite characteristic), but, in general, the latter has primitive elements which are not contained in the Lie algebra. Hence this does not seem to be the appropriate notion if we wish to imitate the
above-mentioned equivalence between ${\mathfrak{M}}_2$ and $\mathsf{Lie}(%
\mathfrak{M})$ we had in case $\mathrm{char}\left( \Bbbk \right) =0$.
\newline
The aim of this note is to provide an appropriate equivalence in case of
prime characteristic. In order to do so, we will use a slightly different approach than the one in
\cite{AM-MonadSym} (as Kharchenko's mentioned lemma is not at hand in prime characteristic). Therefore, recall that a \textit{restricted Lie algebra}
in characteristic $p$ (which is a notion due to Jacobson, see \cite{Jac}) is
a triple $(L,[-,-],{-}^{[p]})$ where $(L,[,])$ is an ordinary $\Bbbk$-Lie algebra,
endowed with a map ${-}^{[p]}: L\to L$ satisfying three conditions. 
These restricted Lie algebras in many respects bear a closer relation to Lie
algebras in characteristic $0$ than ordinary Lie algebras in characteristic $%
p$. 
\\Now, restricted Lie algebras cannot be seen as Lie algebras in some
abelian symmetric monoidal category, at least not to the authors' knowledge.
However, in this short article we show that restricted Lie algebras allow
for an interpretation using monadic decomposition of the functor $P$.
Indeed, Theorem \ref{teo:DecVec} states that one can construct a functor ${\Lambda}: {%
\mathfrak{M}}_2\to \mathsf{{Lie}_{p}}$, $\mathsf{{Lie}_{p}}$ being the
category of restricted Lie algebras over $\Bbbk$, such that $(P_2 \widetilde{\mathfrak{u}}%
, {\Lambda})$ defines an equivalence between $\mathsf{{Lie}_{p}}$ and ${%
\mathfrak{M}}_2$. Here $\widetilde{\mathfrak{u}}$ is the functor computing
the restricted universal enveloping algebra of a restricted Lie algebra. 
\\The article is organized as follows. In the preliminary section we recall
some notation and results concerning monadic decomposition. Along the way,
we address to the interested reader by stating two questions which seem to
be of independent interest. 
\\In the second section, we prove Theorem \ref{teo:DecVec}, using a lemma due to Berger.
\\In the last section, we provide an alternative way to arrive at the conclusion of Theorem \ref{teo:DecVec}, by using so-called adjoint squares. These categorical tools actually allows us to refine our main result. Indeed, in Remark \ref{rem-3.8} we obtain that the functor $\mathcal{P}\widetilde{T}$ is left adjoint to the forgetful functor $H_{\mathsf{{Lie}_{p}}}:\mathsf{{Lie}_{p}}\rightarrow \mathfrak{M}$. The left adjoint of $H_{\mathsf{{Lie}_{p}}}$ already appeared in literature, for some particular base field $\Bbbk$ (e.g. $\Bbbk=\mathbb{Z}_2$), under the name of ``free restricted Lie algebra functor'' and several constructions of this functor can be found. We note that here, in our approach, no additional requirement on $\Bbbk$ is needed other than the finite characteristic. Finally, in Remark \ref{rem-3.10}, using adjoint squares, it is shown that the adjunction $(\widetilde{T}_{2},P_{2}) $ turns out to identify with $(\widetilde{\mathfrak{u}},\mathcal{P})$ via $\Lambda $, $\mathcal{P}$ being the functor
that computes the restricted primitive elements of a bialgebra in
characteristic $p$. 

\section{Preliminary results}

\label{preliminares} In this section, we shall fix some basic notation and
terminology.

\begin{notation}
Throughout this note $\Bbbk $ will denote a field. $\mathfrak{M}$ will
denote the category of vector spaces over $\Bbbk $. Unadorned tensor product
are to taken over $\Bbbk $ unless stated otherwise. \newline
When $X$ is an object in a category $\mathcal{C}$, we will denote the
identity morphism on $X$ by $1_X$ or $X$ for short. For categories $\mathcal{%
C}$ and $\mathcal{D}$, a functor $F:\mathcal{C}\to \mathcal{D}$ will be the
name for a covariant functor; it will only be a contravariant one if it is
explicitly mentioned. By $\mathsf{id}_{\mathcal{C}}$ we denote the identity
functor on $\mathcal{C}$. For any functor $F:\mathcal{C}\to \mathcal{D}$, we
denote $\mathsf{Id}_{F}$ (or sometimes -in order to lighten notation in some
computations- just $F$, if the context doesn't allow for confusion) the
natural transformation defined by $\mathsf{Id}_{FX}=1_{FX}$.
\end{notation}

\begin{claim}
\textbf{Monadic decomposition.} Recall that a \textit{monad} on a category $%
\mathcal{A}$ is a triple $\mathbb{Q}:=\left( Q,m,u\right) $ consisting of a
functor $Q:\mathcal{A}\rightarrow \mathcal{A}$ and natural transformations $%
m:QQ\rightarrow Q$ and $u:\mathcal{A}\rightarrow Q$ satisfying the
associativity and the unitality conditions $m\circ mQ=m\circ Qm$ and $m\circ
Qu=\mathrm{\mathsf{Id}}_{Q}=m\circ uQ.$ An\emph{\ }\textit{algebra} over a
monad $\mathbb{Q}$ on $\mathcal{A}$ (or simply a $\mathbb{Q}$\emph{-algebra}%
) is a pair $\left( X,{\mu }\right) $ where $X\in \mathcal{A}$ and ${\mu }%
:QX\rightarrow X$ is a morphism in $\mathcal{A}$ such that ${\mu }\circ Q{%
\mu }={\mu }\circ mX$ and ${\mu }\circ uX=X.$ A \emph{morphism between two} $%
\mathbb{Q}$-\emph{algebras} $\left( X,{\mu }\right) $ and $\left( X^{\prime
},{\mu }^{\prime }\right) $ is a morphism $f:X\rightarrow X^{\prime }$ in $%
\mathcal{A}$ such that ${\mu }^{\prime }\circ Qf=f\circ {\mu }$. For the
time being, we will denote by ${_{\mathbb{Q}}\mathcal{A}}$ the category of $%
\mathbb{Q}$-algebras and their morphisms. This is the so-called \textit{%
Eilenberg-Moore category} of the monad $\mathbb{Q}$. We denote $_QU:{_Q%
\mathcal{A}}\to\mathcal{A}$ the forgetful functor. When the multiplication
and unit of the monad are clear from the context, we will just write $Q$
instead of $\mathbb{Q}$. \newline
\newline
Let $\left( L:\mathcal{B}\rightarrow \mathcal{A},R:\mathcal{A}\rightarrow
\mathcal{B}\right) $ be an adjunction with unit $\eta $ and counit $\epsilon
$. Then $\left( RL,R\epsilon L,\eta \right) $ is a monad on $\mathcal{B}$
and we can consider the so-called \textit{comparison functor} $K:\mathcal{A}%
\rightarrow {_{RL}\mathcal{B}}$ of the adjunction $\left( L,R\right) $ which
is defined by $KX:=\left( RX,R\epsilon X\right) $ and $Kf:=Rf.$ Note that $%
_{RL}U\circ K=R$.

\begin{definition}
An adjunction $(L:\mathcal{B}\rightarrow \mathcal{A},R:\mathcal{A}%
\rightarrow \mathcal{B})$ is called \textit{monadic} (tripleable in Beck's
terminology \cite[Definition 3]{Beck}) whenever the comparison functor $K:%
\mathcal{A}\rightarrow {_{RL}}\mathcal{B}$ is an equivalence of categories.
A functor $R$ is called \textit{monadic} if it has a left adjoint $L$ such
that the adjunction $(L,R)$ is monadic, see \cite[Definition 3']{Beck}.
\end{definition}

\begin{definition}
(\cite[page 231]{AT}) A monad $\left( Q,m,u\right) $ is called \textit{%
idempotent} if $m$ is an isomorphism. An adjunction $\left( L,R\right) $ is
called \textit{idempotent} whenever the associated monad is idempotent.
\end{definition}

The interested reader can find results on idempotent monads in \cite{AT,MS}.
Here we just note that the fact that $\left( L,R\right) $ being idempotent
is equivalent to requiring that $\eta R$ is a natural isomorphism. The
notion of idempotent monad is tightly connected with the following.

\begin{definition}
\label{def:MonDec} (See \cite[Definition 2.7]{AGM-MonadicLie1}, \cite[%
Definition 2.1]{AHT} and \cite[Definitions 2.10 and 2.14]{MS}) Fix a $N\in
\mathbb{N}$. A functor $R$ is said to have a\emph{\ }\textit{monadic
decomposition of monadic length }$N$\emph{\ }when there exists a sequence $%
\left( R_{n}\right) _{n\leq N}$ of functors $R_{n}$ such that

1) $R_{0}=R$;

2) for $0\leq n\leq N$, the functor $R_{n}$ has a left adjoint functor $%
L_{n} $;

3) for $0\leq n\leq N-1$, the functor $R_{n+1}$ is the comparison functor
induced by the adjunction $\left( L_{n},R_{n}\right) $ with respect to its
associated monad;

4) $L_{N}$ is full and faithful while $L_{n}$ is not full and faithful for $%
0\leq n\leq N-1.$ \newline
A functor $R$ having monadic length $N$ is equivalent to requiring that the
forgetful functor $U_{N,N+1}$ defines an isomorphism of categories and that
no $U_{n,n+1}$ has this property for $\leq n\leq N-1$ (see \cite[Remark 2.4]%
{AGM-MonadicLie1}). \newline
Note that for a functor $R:\mathcal{A}\rightarrow \mathcal{B}$ having a
monadic decomposition of monadic length\emph{\ }$N$, we thus have a diagram
\begin{equation}  \label{diag:MonadicDec}
\xymatrixcolsep{2cm} \xymatrix{\mathcal{A}\ar@<.5ex>[d]^{R_0}&\mathcal{A}%
\ar@<.5ex>[d]^{R_1}\ar[l]_{\mathrm{Id}_{\mathcal{A}}}&\mathcal{A}%
\ar@<.5ex>[d]^{R_2}\ar[l]_{\mathrm{Id}_{\mathcal{A}}}&\cdots
\ar[l]_{\mathrm{Id}_{\mathcal{A}}}\quad\cdots&\mathcal{A}\ar@<.5ex>[d]^{R_N}%
\ar[l]_{\mathrm{Id}_{\mathcal{A}}}\\
\mathcal{B}_0\ar@<.5ex>@{.>}[u]^{L_0}&\mathcal{B}_1\ar@<.5ex>@{.>}[u]^{L_1}
\ar[l]_{U_{0,1}}&\mathcal{B}_2 \ar@<.5ex>@{.>}[u]^{L_2}
\ar[l]_{U_{1,2}}&\cdots\ar[l]_{U_{2,3}}\quad\cdots&\mathcal{B}_N%
\ar@<.5ex>@{.>}[u]^{L_N}\ar[l]_{U_{N-1,N}} }
\end{equation}

where $\mathcal{B}_{0}=\mathcal{B}$ and, for $1\leq n\leq N,$

\begin{itemize}
\item $\mathcal{B}_{n}$ is the category of $\left( R_{n-1}L_{n-1}\right) $%
-algebras ${_{R_{n-1}L_{n-1}}\mathcal{B}}_{n-1}$;

\item $U_{n-1,n}:\mathcal{B}_{n}\rightarrow \mathcal{B}_{n-1}$ is the
forgetful functor ${_{R_{n-1}L_{n-1}}}U$.
\end{itemize}

We will denote by $\eta _{n}:{\mathsf{id}}_{\mathcal{B}_{n}}\rightarrow
R_{n}L_{n}$ and $\epsilon _{n}:L_{n}R_{n}\rightarrow {\mathsf{id}}_{\mathcal{%
A}}$ the unit, resp. counit of the adjunction $\left( L_{n},R_{n}\right) $
for $0\leq n\leq N$. Note that one can introduce the forgetful functor $%
U_{m,n}:\mathcal{B}_{n}\rightarrow \mathcal{B}_{m}$ for all $m\leq n $ with $%
0\leq m,n\leq N$.
\end{definition}

We recall the following from \cite{AGM-MonadicLie1}:

\begin{proposition}
\cite[Proposition 2.9]{AGM-MonadicLie1} \label{pro:idemmonad2}Let $\left( L:%
\mathcal{B}\rightarrow \mathcal{A},R:\mathcal{A}\rightarrow \mathcal{B}%
\right) $ be an idempotent adjunction. Then $R:\mathcal{A}\rightarrow
\mathcal{B}$ has a\emph{\ }monadic decomposition of monadic length at most $%
1 $.
\end{proposition}

Letting $\Bbbk$ be a field, and $B$ a $\Bbbk$-bialgebra, the set $P({B})$ of
primitive elements of ${B}$ is defined as
\begin{equation*}
P({B})=\{x\in {B} | \Delta(x)=1\otimes x+x\otimes 1\},
\end{equation*}
where $\Delta$ is the comultiplication of $B$. $P({B})$ forms a $\Bbbk$%
-vector space, yielding a functor
\begin{equation}  \label{primitive}
P:\mathsf{Bialg}(\mathfrak{M})\to \mathfrak{M}
\end{equation}
Theorem 3.4 from loc. cit. asserts that the functor $P$ has monadic
decomposition at most 2, by showing that the comparison functor $P_1$ of the
adjunction $(\widetilde{T},P)$ admits a left adjoint $\widetilde{T}_1$ such
that the adjunction $(\widetilde{T}_1,P_1)$ is idempotent. For the sake of
completeness, we recall here that $\widetilde{T}$ is the functor from $%
\mathfrak{M}$ to $\mathsf{Bialg}(\mathfrak{M})$, assigning to any vector
space $V$ the tensor algebra $T(V)$ (which can be endowed with a bialgebra
structure $\widetilde{T}(V)$, as is known).
\newline
Intriguingly, it is not known to the authors whether the bound provided by
this above-mentioned Theorem 3.4 is sharp. It would thus be satisfying to
have an answer to the following question -of independent interest- the
interested reader is evidently invited to think about.

\begin{question}
Is the functor $\widetilde{T}_1$ fully faithful?
\end{question}

As mentioned in the Introduction, it is known -by combining Theorems 7.2 and
8.1 from \cite{AM-MonadSym}- that in case $\mathrm{char}(\Bbbk)=0$, the
category ${\mathfrak{M}}_{2}$ is equivalent to the category of $\Bbbk$-Lie
algebras. It is the aim of this note to handle the case of finite
characteristic. Before doing so, we would like to round off this preliminary
section by the following.

\begin{definition}
\label{def:comparable}We say that a functor $R$ is\emph{\ \textit{comparable}
}whenever there exists a sequence $\left( R_{n}\right) _{n\in \mathbb{N}}$
of functors $R_{n}$ such that $R_{0}=R$ and, for $n\in \mathbb{N}$,

1) the functor $R_{n}$ has a left adjoint functor $L_{n}$;

2) the functor $R_{n+1}$ is the comparison functor induced by the adjunction
$\left( L_{n},R_{n}\right) $ with respect to its associated monad.

In this case we have a diagram as (\ref{diag:MonadicDec}) but not
necessarily stationary. Hence we can consider the forgetful functors $%
U_{m,n}:\mathcal{B}_{n}\rightarrow \mathcal{B}_{m}$ for all $m\leq n$ with $%
m,n\in \mathbb{N}$.
\end{definition}

\begin{remark}
Fix a $N\in \mathbb{N}$. A functor $R$ having a\emph{\ }monadic
decomposition of monadic length $N$ is comparable, see \cite[Remark 2.10]%
{AGM-MonadicLie1}.
\end{remark}

By the proof of Beck's Theorem \cite[Proof of Theorem 1]{Beck}, one gets the
following result.

\begin{lemma}[\protect\cite{AM-MonadSym}]
\label{lem: coequalizers} Let $\mathcal{A}$ be a category such that, for any
(reflexive) pair $\left( f,g\right) $ (\cite[3.6, page 98]{BW-TTT}) where $%
f,g:X\rightarrow Y$ are morphisms in $\mathcal{A}$, one can choose a
specific coequalizer. Then the comparison functor $K:\mathcal{A}\rightarrow {%
_{RL}\mathcal{B}}$ of an adjunction $\left( L,R\right) $ is a right adjoint.
Thus any right adjoint $R:\mathcal{A}\rightarrow {\mathcal{B}}$ is
comparable.
\end{lemma}

Let $\Bbbk$ again be a field of characteristic zero. Dually to $\Bbbk$-Lie
algebras, recall that $\Bbbk$-Lie coalgebras, as introduced by Michaelis in
\cite{Michaelis-LieCoalg}, are precisely Lie algebras in the abelian
symmetric monoidal category $({\mathfrak{M}}^{op}, {\otimes}^{op}, \Bbbk)$,
where associativity and unit constraints are taken to be trivial. \newline
$Q(B)$ denotes the $\Bbbk$-Lie coalgebra of indecomposables of a $\Bbbk$%
-bialgebra, more precisely, $Q(B)=I/I^2$, where $I=\ker\varepsilon$, $%
\varepsilon$ being the counit of $B$. This construction is functorial and,
composed with the forgetful functor $F$ from the category of $\Bbbk$-Lie
coalgebras to ${\mathfrak{M}}$, yields the following functor
\begin{equation*}
Q: \mathsf{Bialg}(\mathfrak{M})\to {\mathfrak{M}}.
\end{equation*}
In \cite[page 18]{Michaelis-LieCoalg}, it is asserted that a right adjoint
for the functor $F$ is provided by the functor $L^c$ that computes the
so-called ``cofree Lie coalgebra'' on a vector space. Finally, $Q$ has a
right adjoint given by the cofree coalgebra functor $\widetilde{T}^c$. In
fact $Q=F\circ \widetilde{Q}$, where $\widetilde{Q}$ is the functor sending
a bialgebra $B$ to the $\Bbbk$-Lie coalgebra $Q(B)$, while, by \cite[page 24]%
{Michaelis-LieCoalg}, we have $\widetilde{T}^c=\widetilde{U}^c_H\circ L^c$,
where $\widetilde{U}^c_H(C)$ is the universal coenveloping bialgebra of a
Lie coalgebra $C$ (by the bialgebra version of \cite[Theorem, page 37]%
{Michaelis-LieCoalg}, the functor $\widetilde{U}^c_H$ is right adjoint to $%
\widetilde{Q}$). Lemma \ref{lem: coequalizers} now guarantees that the
functor
\begin{equation*}
Q^{op}: \mathsf{Bialg}(\mathfrak{M}^{op})\to {\mathfrak{M}^{op}}
\end{equation*}
is comparable. Strangely enough, at the moment, we don't have an answer to
the following question, which is -again- of independent interest in the
authors' opinion.

\begin{question}
Does the functor $Q^{op}$ also have monadic decomposition length at most 2?
If so, is the resulting category ${\mathfrak{M}}_2$ equivalent to the
category of $\Bbbk$-Lie coalgebras?
\end{question}
\end{claim}


\section{The category of restricted Lie algebras}

For the sake of the reader's comfort, we include a result due to Berger, here presented in a slightly different form.

\begin{lemma}[{cf. \protect\cite[Lemma 1.2]{Ber}}]
\label{lem:Ber} Consider the following diagram%
\begin{equation*}
\xymatrixrowsep{15pt}\xymatrixcolsep{20pt} \xymatrix{&\mathcal{E}^\prime
\ar@<.5ex>[dl]^\Psi\ar@<.5ex>[dr]^{\Psi^\prime}\\
\mathcal{M}\ar@{.>}@<.5ex>[ur]^\Phi\ar[dr]_U&&\mathcal{M}^\prime%
\ar@{.>}@<.5ex>[ul]^{\Phi^\prime}\ar@<.5ex>[dl]^{U^\prime}\\ &\mathcal{E}}
\end{equation*}%
where

\begin{itemize}
\item $U^{\prime }\circ \Psi ^{\prime }=U\circ \Psi ,$

\item $U$ and $U^{\prime }$ are conservative,

\item $\Psi $ and $\Psi ^{\prime }$ are coreflections i.e. functors having
fully faithful left adjoints $\Phi $ and $\Phi ^{\prime }$ respectively
(i.e. the units $\eta :\mathrm{id}_{\mathcal{M}}\rightarrow \Psi \Phi $ and $%
\eta ^{\prime }:\mathrm{id}_{\mathcal{M}^{\prime }}\rightarrow \Psi ^{\prime
}\Phi ^{\prime }$ are invertible).
\end{itemize}

Then $\left( \Psi ^{\prime }\Phi ,\Psi \Phi ^{\prime }\right) $ is an
adjoint equivalence of categories with unit $\widehat{\eta }$ and counit $%
\widehat{\epsilon }$ defined by
\begin{equation*}
\left( \widehat{\eta }\right) ^{-1}:=\Psi \Phi ^{\prime }\Psi ^{\prime }\Phi
\overset{\Psi \epsilon ^{\prime }\Phi }{\longrightarrow }\Psi \Phi \overset{%
\eta ^{-1}}{\longrightarrow }\mathrm{id}_{\mathcal{M}},
\end{equation*}%
\begin{equation*}
\widehat{\epsilon }:=\Psi ^{\prime }\Phi \Psi \Phi ^{\prime }\overset{\Psi
^{\prime }\epsilon \Phi ^{\prime }}{\longrightarrow }\Psi ^{\prime }\Phi
^{\prime }\overset{\left( \eta ^{\prime }\right) ^{-1}}{\longrightarrow }%
\mathrm{id}_{\mathcal{M}^{\prime }}.
\end{equation*}
\end{lemma}

Fix an arbitrary field $\Bbbk $ such that $\mathrm{char}\left( \Bbbk \right)$
is a prime $p$ and recall that $\mathfrak{M}$ denotes the category of vector
spaces over $\Bbbk $.

\begin{definition}
(due to Jacobson, see \cite[page 210]{Jac}) A \textit{restricted Lie algebra
over $\Bbbk$} (also called \textit{$p$-Lie algebra} by some authors) is a
triple $(L,[-,-], {-}^{[p]})$ consisting of a (ordinary) Lie algebra $%
(L,[-,-])$ (i.e. a $\Bbbk$-vector space $L$ endowed with a $\Bbbk$-bilinear
map $[-,-]$ satisfying the antisymmetry and Jacobi condition) and a
(set-theoretical) map ${-}^{[p]}:L\to L$ satisfying%
\begin{eqnarray*}
(\alpha x)^{[p]} &=& {\alpha^{p}x^{[p]}} \text{~for~all~} x\in L, \alpha\in
\Bbbk; \\
\mathrm{ad}(x^{[p]})&=&(\mathrm{ad}(x))^{p} \text{~for~all~} x\in L; \\
(x+y)^{[p]}&=&x^{[p]}+y^{[p]}+ s(x,y) \text{~for~all~} x,y\in L,
\end{eqnarray*}%
where $\mathrm{ad}$ is the adjoint representation of $L$;
\begin{equation*}
\mathrm{ad}:L\to \mathrm{End}(L), x\mapsto {\mathrm{ad}}_{x} ~~\mathrm{where}%
~{\mathrm{ad}}_{x}(y)=[x,y],
\end{equation*}
and $s(x,y)= \sum_{i=1}^{p-1}\frac{s_{i}(x,y)}{i}$, where $s_{i}(x,y)$ is
the coefficient of ${\beta}^{i-1}$ in the expansion of $({\mathrm{ad}}(\beta
x+y))^{p-1}(x)$. \newline
A map $f:(L,[-,-], {-}^{[p]})\to (L^{\prime },[-,-]^{\prime }, {-}%
^{[p^{\prime }]})$ is a morphism of restricted Lie algebras if $f$ is a
morphism of (ordinary) Lie algebras $f:(L,[-,-])\to (L^{\prime
},[-,-]^{\prime })$ such that $f(x^{[p]})=(f(x))^{[p^{\prime }]},$ for all $%
x\in L$. \newline
The category of restricted Lie algebras with their morphisms will be denoted
by $\mathsf{{Lie}_{p}}$.
\end{definition}

There is an adjunction $\left({\widetilde{\mathfrak{u}}}:\mathsf{{Lie}_{p}}%
\rightarrow \mathsf{Bialg}(\mathfrak{M}),\mathcal{P}:\mathsf{Bialg}(%
\mathfrak{M})\rightarrow \mathsf{{Lie}_{p}}\right) $, given by the following
functors (see \cite{Grunenfelder-Phd} or \cite[Appendix]%
{Michaelis-ThePrimitives}, e.g.):

\begin{itemize}
\item $\widetilde{\mathfrak{u}}: \mathsf{{Lie}_{p}}\to \mathsf{Bialg}(%
\mathfrak{M})$; the restricted universal enveloping algebra functor. \newline
Explicitly, $\widetilde{\mathfrak{u}}(L,[-,-],{-}^{[p]})=\frac{\widetilde{U}%
(L,[-,-])}{I}$, where $I$ is the ideal in $\widetilde{U}(L,[-,-])$ generated
by elements of the form $x^p-x^{[p]}$.

\item $\mathcal{P}: \mathsf{Bialg}(\mathfrak{M})\to \mathsf{{Lie}_{p}}$; the
restricted primitive functor. \newline
Explicitly, for $B\in \mathsf{Bialg}(\mathfrak{M})$, the space $P(B)$
becomes a Lie algebra for the commutator bracket $[-,-]$ and can moreover be
endowed with the map ${-}^{[p]}$ sending an element $x\in L$ to $x^p$ such
that $\mathcal{P}(B):=(P(B),[-,-], {-}^{[p]})$ becomes a restricted Lie
algebra.
\end{itemize}

We denote by $\widetilde{\eta }_{\mathrm{L}}$ the unit and by $\widetilde{%
\epsilon }_{\mathrm{L}}$ the counit of the adjunction $( \widetilde{%
\mathfrak{u}},\mathcal{P}) $. By \cite[Theorem 6.11(1)]%
{Milnor-Moore}, we know that $\widetilde{\eta }_{\mathrm{L}}:{\mathsf{id%
}}_{{\mathsf{{Lie}_{p}}}}\rightarrow \mathcal{P}\widetilde{\mathfrak{u}}$ is
an isomorphism, see also \cite[Theorem 1]{Jacp}. We also use the notation $H_{\mathsf{{Lie}_{p}}%
}:\mathsf{{Lie}_{p}}\rightarrow \mathfrak{M}$ for the forgetful functor. We
obviously have that $H_{\mathsf{{Lie}_{p}}}\mathcal{P}=P:\mathsf{Bialg}(%
\mathfrak{M})\rightarrow \mathfrak{M}$ is the usual primitive functor (cf. (%
\ref{primitive})). \newline
Before stating the main result, we notice that in case we wish to stress the
algebra nature of objects and morphisms in $\mathsf{Alg}(\mathfrak{M})$,
resp. the bialgebra nature of objects and morphisms in $\mathsf{Bialg}(%
\mathfrak{M})$, we will do so by simply overlining, resp. over and
underlining things. Please mind as well that we denote $\eta$ (resp. $%
\widetilde{\eta}$) and $\epsilon$ (resp. $\widetilde{\epsilon}$) the unit
and counit of the adjunction $(T,\Omega)$ (resp. $(\widetilde{T},P)$).

\begin{theorem}
\label{teo:DecVec}We have the following diagram.
\begin{equation}
\xymatrixcolsep{1cm}\xymatrixrowsep{0.40cm}\xymatrix{\Bialg({\mathfrak{M}})%
\ar@<.5ex>[dd]^{P}&&\Bialg({\mathfrak{M}})\ar@<.5ex>[dd]^{P_1}|(.30)\hole%
\ar[ll]_{\id_{\Bialg({\mathfrak{M})}}}&&\Bialg({\mathfrak{M}})%
\ar@<.5ex>[dd]^{P_2}\ar[ll]_{\id_{\Bialg({\mathfrak{M}})}}\ar[dl]|{\id_{%
\Bialg({\mathfrak{M}})}}\\
&&&\Bialg({\mathfrak{M}})\ar@<.5ex>[dd]^(.30){\mathcal{P}}\ar[ulll]^(.70){%
\id_{\Bialg({\mathfrak{M}})}}\\
\mathfrak{M}\ar@<.5ex>@{.>}[uu]^{\widetilde{T}}&&\mathfrak{M}_1%
\ar@<.5ex>@{.>}[uu]^{ \widetilde{T}_1}|(.70)\hole
\ar[ll]_{U_{0,1}}&&\mathfrak{M}_2 \ar@<.5ex>@{.>}[uu]^{\widetilde{T}_2}
\ar[ll]_(.30){U_{1,2}}|\hole \ar[dl]^{\Lambda} \\
&&&\RLie\ar@<.5ex>@{.>}[uu]^(.70){\widetilde{\mathfrak{u}}}\ar[ulll]^{H_{%
\RLie}}}  \label{diag:Lambda}
\end{equation}

The functor $P$ is comparable so that we can use the notation of Definition %
\ref{def:comparable}. There is a functor $\Lambda :\mathfrak{M}%
_{2}\rightarrow \mathsf{{Lie}_{p}}$ such that $\Lambda \circ P_{2}=\mathcal{P%
}$ and $H_{\mathsf{{Lie}_{p}}}\circ \Lambda =U_{0,2}.$

\begin{itemize}
\item The adjunction $( \widetilde{T}_{1},P_{1}) $ is idempotent, we can
choose $\widetilde{T}_{2}:=\widetilde{T}_{1}U_{1,2}$, $\pi _{2}={\mathsf{Id}}%
_{\widetilde{T}_{2}}$ and $\widetilde{T}_{2}$ is full and faithful i.e. $%
\widetilde{\eta }_{2}$ is an isomorphism. The functor $P$ has a monadic
decomposition of monadic length at most two.


\item The pair $(P_{2}\widetilde{\mathfrak{u}},\Lambda )$ is an adjoint
equivalence of categories with unit $\widetilde{\eta }_{\mathrm{L}}$ and
counit $\left( \widetilde{\eta }_{2}\right) ^{-1}\circ P_{2}\left( \widetilde{%
\epsilon }_{\mathrm{L}}\widetilde{T}_{2}\circ \widetilde{\mathfrak{u}}%
\Lambda \widetilde{\eta }_{2}\right)$.
\end{itemize}
\end{theorem}

\begin{proof}
By \cite[Theorem 3.4]{AGM-MonadicLie1}, the functor $P$ has monadic
decomposition of monadic length at most $2$. Moreover, the adjunction $(
\widetilde{T}_{1},P_{1}) $ is idempotent and we can define a functor $%
\Lambda :\mathfrak{M}_{2}\rightarrow \mathsf{{Lie}_{p}}$. Indeed, letting $%
V_{2}=\left( \left( V_{0},\mu _{0}\right) ,\mu _{1}\right)$ be an object in $%
\mathfrak{M}_{2}$, we can define an object $\Lambda V_{2}\in \mathsf{{Lie}%
_{p}}$ as follows:
\begin{equation*}
\Lambda V_{2}=\left( V_{0},\left[ -,-\right] ,-^{\left[ p\right] }\right),
\end{equation*}
where $\left[ -,-\right] :V_{0}\otimes V_{0}\rightarrow V_{0}$ is defined by
setting $\left[ x,y\right] :=\mu _{0}\left( x\otimes y-y\otimes x\right)$,
for every $x,y\in V_{0}$, while $-^{\left[ p\right] }:V_{0}\rightarrow V_{0}$
is defined by setting $x^{\left[ p\right] }:=\mu _{0}\left( x^{\otimes
p}\right) ,\,$for every $x\in V_{0}.$

Let $f_{2}:V_{2}\rightarrow V_{2}^{\prime }$ be a morphism in $\mathfrak{M}%
_{2}$ and set $f_{1}:=U_{1,2}f_{2}$ and $f_{0}:=U_{0,1}f_{1}$. Then, for
every $x,y\in V_{0}$
\begin{eqnarray*}
f_{0}\left( \left[ x,y\right] \right) &=&f_{0}\mu _{0}\left( x\otimes
y-y\otimes x\right) =\mu _{0}^{\prime }\left( P\widetilde{T}f_{0}\right)
\left( x\otimes y-y\otimes x\right) \\
&=&\mu _{0}^{\prime }\left( \left( \widetilde{T}f_{0}\right) \left( x\otimes
y-y\otimes x\right) \right) =\mu _{0}^{\prime }\left[ \left( f_{0}\left(
x\right) \otimes f_{0}\left( y\right) -f_{0}\left( y\right) \otimes
f_{0}\left( x\right) \right) \right] \\
&=&\left[ f_{0}\left( x\right) ,f_{0}\left( y\right) \right] ^{\prime }
\end{eqnarray*}%
and
\begin{eqnarray*}
f_{0}\left( x^{\left[ p\right] }\right) &=&f_{0}\mu _{0}\left( x^{\otimes
p}\right) =\mu _{0}^{\prime }\left( P\widetilde{T}f_{0}\right) \left(
x^{\otimes p}\right) =\mu _{0}^{\prime }\left( \left( \widetilde{T}%
f_{0}\right) \left( x^{\otimes p}\right) \right) \\
&=&\mu _{0}^{\prime }\left( f_{0}\left( x\right) ^{\otimes p}\right)
=f_{0}\left( x\right) ^{\left[ p\right] ^{\prime }}.
\end{eqnarray*}%
Thus $f_{2}$ induces a unique morphism $\Lambda f_{2}$ such that $H_{\mathsf{%
{Lie}_{p}}}(\Lambda f_{2})=f_{0}.$ This defines a functor $\Lambda :%
\mathfrak{M}_{2}\rightarrow \mathsf{{Lie}_{p}}$, as claimed above. By
construction $H_{{\mathsf{{Lie}_{p}}}}\circ \Lambda =U_{0,2}.$ Moreover,
\begin{equation*}
\left( \Lambda \circ P_{2}\right) \left( \underline{\overline{B}}\right)
=\Lambda \left( P_{2}\left( \underline{\overline{B}}\right) \right) .
\end{equation*}%
In order to proceed, we have to compute $\Lambda \left( P_{2}\left(
\underline{\overline{B}}\right) \right)$. We have that
\begin{equation*}
P_{1}\left( \underline{\overline{B}}\right) =\left( P\left( \underline{%
\overline{B}}\right) ,P{\widetilde{\epsilon }}_{\underline{\overline{B}}}:P%
\widetilde{T}P\left( \underline{\overline{B}}\right) \rightarrow P\left(
\underline{\overline{B}}\right) \right)
\end{equation*}
so the brackets $\left[ -,-\right] $ and $-^{\left[ p\right] }$ are, for
every $x,y\in P\left( \underline{\overline{B}}\right) $, given by the
following:
\begin{eqnarray*}
\left[ x,y\right] &=&\left( P{\widetilde{\epsilon }}_{\underline{\overline{B}%
}}\right) \left( x\otimes y-y\otimes x\right) =\left( {\widetilde{\epsilon }}%
_{ \underline{\overline{B}}}\right) \left( x\otimes y-y\otimes x\right)
=\left( {\epsilon}_{ \overline{B}}\right) \left( x\otimes y-y\otimes
x\right) =xy-yx, \\
x^{\left[ p\right] } &=&\left( P{\widetilde{\epsilon }}_{\underline{%
\overline{B}}}\right) \left( x^{\otimes p}\right) =\left( {\widetilde{%
\epsilon }}_{\underline{\overline{B}}}\right) \left( x^{\otimes p}\right)
=\left( {\epsilon}_{ \overline{B}}\right) \left( x^{\otimes p}\right) =x^{p}
\end{eqnarray*}%
so that $\left( \Lambda \circ P_{2}\right) \left( \underline{\overline{B}}%
\right) =\mathcal{P}\left( \underline{\overline{B}}\right)$. For the
morphisms, we have%
\begin{equation*}
\left( H_{{\mathsf{{Lie}_{p}}}}\right) \Lambda P_{2}\left( \underline{%
\overline{f}}\right) =U_{0,2}P_{2}\left( \underline{\overline{f}}\right)
=P\left( \underline{\overline{f}}\right) =\left( H_{{\mathsf{{Lie}_{p}}}%
}\right) \mathcal{P}\left( \underline{\overline{f}}\right) .
\end{equation*}%
Since $H_{{\mathsf{{Lie}_{p}}}}$ is faithful, we conclude that $\Lambda
P_{2}\left( \underline{\overline{f}}\right) =\mathcal{P}\left( \underline{%
\overline{f}}\right) $ and hence $\Lambda \circ P_{2}=\mathcal{P}$.
\\Now, since the adjunction $(\widetilde{T}_{1},P_{1})$ is idempotent, by \cite%
[Proposition 2.3]{AGM-MonadicLie1}, we can choose $\widetilde{T}_{2}:=%
\widetilde{T}_{1}U_{1,2}$ with $\widetilde{\eta }_{1}U_{1,2}=U_{1,2}%
\widetilde{\eta }_{2}$ and $\widetilde{\epsilon }_{1}=\widetilde{\epsilon }%
_{2}$. Since $\widetilde{T}_{2}$ is full and faithful, we have that $%
\widetilde{\eta }_{2}$ is an isomorphism. We already observed that $\widetilde{\eta }_{\mathrm{L}}:{\mathsf{id%
}}_{{\mathsf{{Lie}_{p}}}}\rightarrow \mathcal{P}\widetilde{\mathfrak{u}}$ is
an isomorphism.

Since $U_{0,2}=H_{\mathsf{{Lie}_{p}}}\circ \Lambda $ and $U_{0,2}$ is
conservative, so is $\Lambda .$
We can apply Lemma \ref{lem:Ber} to the following diagram%
\begin{equation*}
\xymatrixrowsep{15pt}\xymatrixcolsep{20pt} \xymatrix{&\mathsf{Bialg}(%
\mathfrak{M}) \ar@<.5ex>[dl]^{\mathcal{P}}\ar@<.5ex>[dr]^{P_{2}}\\
\mathsf{{Lie}_{p}}\ar@{.>}@<.5ex>[ur]^{\widetilde{\mathfrak{u}}}\ar[dr]_{H_{%
\mathsf{{Lie}_{p}}}}&&{\mathfrak{M}}_{2}\ar@{.>}@<.5ex>[ul]^{%
\widetilde{T}_{2}}\ar@<.5ex>[dl]^{U_{0,2}}\\ &\mathfrak{M}}
\end{equation*}%
to deduce that $\left( P_{2}\widetilde{\mathfrak{u}},\mathcal{P}\widetilde{T}%
_{2}\right) $ is an adjoint equivalence with unit $\widehat{\eta }$ and counit
$\widehat{\epsilon }$ defined by  $\left( \widehat{\eta }%
\right) ^{-1}:=\widetilde{\eta }_{\mathrm{L}}^{-1}\circ \mathcal{P}%
\widetilde{\epsilon }_{2}\widetilde{\mathfrak{u}}$ and $\widehat{%
\epsilon }:=\widetilde{\eta }_{2}^{-1}\circ P_{2}\widetilde{\epsilon }_{%
\mathrm{L}}\widetilde{T}_{2}.$ By the first part of the proof, we have $\mathcal{P}=\Lambda P_{2}.$ Thus we can use the
isomorphism $\Lambda \widetilde{\eta }_{2}:\Lambda \rightarrow \Lambda P_{2}%
\widetilde{T}_{2}=\mathcal{P}\widetilde{T}_{2}$ to replace $\mathcal{P}%
\widetilde{T}_{2}$ by $\Lambda $ in the adjunction.
Thus we obtain that the pair $(P_{2}\widetilde{\mathfrak{u}},\Lambda )$ is an equivalence
of categories with unit $\widetilde{\eta }_{\mathrm{L}}$ and counit $\widetilde{\eta }_{2}^{-1}\circ P_{2}\left( \widetilde{\epsilon }_{\mathrm{L}}\widetilde{T}_{2}\circ \widetilde{\mathfrak{u}}%
\Lambda \widetilde{\eta }_{2}\right) $ by the following computations
\begin{gather*}
\left( \widehat{\eta }\right) ^{-1}\circ \Lambda \widetilde{\eta }_{2}P_{2}%
\widetilde{\mathfrak{u}}=\widetilde{\eta }_{\mathrm{L}}^{-1}\circ \mathcal{P}%
\widetilde{\epsilon }_{2}\widetilde{\mathfrak{u}}\circ \Lambda \widetilde{%
\eta }_{2}P_{2}\widetilde{\mathfrak{u}}=\widetilde{\eta }_{\mathrm{L}%
}^{-1}\circ \Lambda P_{2}\widetilde{\epsilon }_{2}\widetilde{\mathfrak{u}}%
\circ \Lambda \widetilde{\eta }_{2}P_{2}\widetilde{\mathfrak{u}}=\widetilde{%
\eta }_{\mathrm{L}}^{-1},\\
\widehat{\epsilon }\circ P_{2}\widetilde{\mathfrak{u}}\Lambda \widetilde{%
\eta }_{2}=\widetilde{\eta }_{2}^{-1}\circ P_{2}\widetilde{\epsilon }_{%
\mathrm{L}}\widetilde{T}_{2}\circ P_{2}\widetilde{\mathfrak{u}}\Lambda
\widetilde{\eta }_{2}=\widetilde{\eta }_{2}^{-1}\circ P_{2}\left( \widetilde{%
\epsilon }_{\mathrm{L}}\widetilde{T}_{2}\circ \widetilde{\mathfrak{u}}%
\Lambda \widetilde{\eta }_{2}\right) .
\end{gather*}
\end{proof}

\section{An alternative approach via adjoint squares}

The main aim of this section is to give an alternative approach to Theorem \ref{teo:DecVec} by means of some results that -in our opinion- could have an interest in their own right.

\begin{definition}
\label{def:AdjSquare}Recall from \cite[Definition I,6.7, page 144]{Gray-FormalCat}
that an \textit{adjoint square} consists of a (not necessarily commutative)
diagram of functors as depicted below ($(L,R)$ and $(L^{\prime },R^{\prime
}) $ being adjunctions with units $\eta$ resp. $\eta^{\prime }$ and counits $%
\epsilon$ resp. $\epsilon^{\prime }$) together with a matrix of natural
transformations ``inside'':
\begin{equation}
\begin{array}{ccc}
\xymatrixrowsep{25pt}\xymatrixcolsep{3cm} \xymatrix{\ar@{} [dr]
|{\left(\begin{array}{cc} \zeta _{11} & \zeta _{12} \\ \zeta _{21} & \zeta
_{22}\end{array}\right)}
\mathcal{A}\ar[r]^F\ar@<.5ex>[d]^R&\mathcal{A}^{\prime
}\ar@<.5ex>[d]^{R^\prime}\\
\mathcal{B}\ar@<.5ex>[u]^{L}\ar[r]_G&\mathcal{B}^{\prime
}\ar@<.5ex>[u]^{L^\prime}} &  &
\end{array}%
\qquad
\begin{array}{cc}
\zeta _{11}:L^{\prime }G\rightarrow FL, & \zeta _{12}:L^{\prime
}GR\rightarrow F, \\
\zeta _{21}:G\rightarrow R^{\prime }FL, & \zeta _{22}:GR\rightarrow
R^{\prime }F,%
\end{array}
\label{diag:AdjSquare}
\end{equation}%
These ingredients are required to be subject to the following equalities:%
\begin{eqnarray}
\zeta _{11} &=&\zeta _{12}L\circ L^{\prime }G\eta =\epsilon ^{\prime
}FL\circ L^{\prime }\zeta _{21}=\epsilon ^{\prime }FL\circ L^{\prime }\zeta
_{22}L\circ L^{\prime }G\eta ,  \label{Adj1} \\
\zeta _{12} &=&F\epsilon \circ \zeta _{11}R=\epsilon ^{\prime }F\epsilon
\circ L^{\prime }\zeta _{21}R=\epsilon ^{\prime }F\circ L^{\prime }\zeta
_{22},  \label{Adj2} \\
\zeta _{21} &=&R^{\prime }\zeta _{11}\circ \eta ^{\prime }G=R^{\prime }\zeta
_{12}L\circ \eta ^{\prime }G\eta =\zeta _{22}L\circ G\eta ,  \label{Adj3} \\
\zeta _{22} &=&R^{\prime }F\epsilon \circ R^{\prime }\zeta _{11}R\circ \eta
^{\prime }GR=R^{\prime }\zeta _{12}\circ \eta ^{\prime }GR=R^{\prime
}F\epsilon \circ \zeta _{21}R.  \label{Adj4}
\end{eqnarray}%
We call such natural transformations \textit{transposes} of each other. If
only one of the entries of the matrix is given, its transposes can be
defined by means of the equalities above.
\end{definition}

\begin{example}
\label{ex:Beck}Let $\left( L:\mathcal{B}\rightarrow \mathcal{A},R:\mathcal{A}%
\rightarrow \mathcal{B}\right) $ be an adjunction with unit $\eta $ and
counit $\epsilon $. Assume that the comparison functor $R_{1}:\mathcal{A}%
\rightarrow \mathcal{B}_{1}$ has a left adjoint $L_{1}$ with unit $\eta _{1}$
and counit $\epsilon _{1}$. We then have an adjoint square
\begin{equation*}
\begin{array}{ccc}
\xymatrixrowsep{25pt}\xymatrixcolsep{3cm}\xymatrix{\ar@{} [dr]
|{\left(\begin{array}{cc} \pi _{11}& \pi _{12} \\ \pi _{21} & \pi
_{22}\end{array}\right)}
\mathcal{A}\ar[r]^{\id_{\Aa}}\ar@<.5ex>[d]^{R_1}&\mathcal{A}%
\ar@<.5ex>[d]^{R}\\
{\mathcal{B}}_{1}\ar@<.5ex>[u]^{L_{1}}\ar[r]_G&\mathcal{B}\ar@<.5ex>[u]^{L}}
&  &
\end{array}%
\end{equation*}%
%
%
%
%
%
where
\begin{equation*}
\pi _{11}\overset{(\ref{Adj1})}{=}\epsilon L_{1}\circ LU_{01}\eta
_{1}:LU_{01}\rightarrow L_{1}~\hspace{0.2cm}\text{and}~\hspace{0.2cm}\pi
_{22}=\mathrm{\mathsf{Id}}_{U_{01}R_{1}}
\end{equation*}%
so that, by the proof of \cite[Theorem 1]{Beck}, $\pi _{11}$ is the
canonical projection defining $L_{1}.$ More explicitly, for every $\left(
B,\mu \right) \in \mathcal{B}_{1}$ we have the following coequalizer of a
reflexive pair in $\mathcal{A}$%
\begin{equation*}
\xymatrix{ LRLB \ar@<.5ex>[rr]^-{L\mu} \ar@<-.5ex>[rr]_-{\epsilon_{LB}}&&
LB=LU_{01}\left( B,\mu \right) \ar@<.5ex>[rr]^-{{\pi}_{11} (B,\mu)}
&&L_{1}\left( B,\mu \right)}.
\end{equation*}%
%
%
%
%
%
As a consequence we get%
\begin{eqnarray}
&&\epsilon _{1}\circ \pi _{11}R_{1}\overset{(\ref{Adj2})}{=}\epsilon ,
\label{Eps1} \\
&&R\pi _{11}\circ \eta U_{01}\overset{(\ref{Adj3})}{=}U_{01}\eta _{1}.
\notag
\end{eqnarray}
\end{example}

\begin{remark}
\label{rem:orizcomp}Given two adjoint squares%
\begin{equation*}
\begin{array}{ccc}
\xymatrixrowsep{25pt}\xymatrixcolsep{3cm} \xymatrix{\ar@{} [dr]
|{\left(\begin{array}{cc} \zeta _{11} & \zeta _{12} \\ \zeta _{21} & \zeta
_{22}\end{array}\right)}
\mathcal{A}\ar[r]^F\ar@<.5ex>[d]^R&\mathcal{A}^{\prime
}\ar@<.5ex>[d]^{R^\prime}\\
\mathcal{B}\ar@<.5ex>[u]^{L}\ar[r]_G&\mathcal{B}^{\prime
}\ar@<.5ex>[u]^{L^\prime}} &  &
\end{array}
\quad \text{and}\quad
\begin{array}{ccc}
\xymatrixrowsep{25pt}\xymatrixcolsep{3cm} \xymatrix{\ar@{} [dr]
|{\left(\begin{array}{cc} {{\zeta}'}_{11} & {{\zeta}'}_{12} \\
{{\zeta}'}_{21} & {{\zeta}'}_{22}\end{array}\right)}
{\mathcal{A}}'\ar[r]^{F'}\ar@<.5ex>[d]^{R'}&{{\mathcal{A}}''}%
\ar@<.5ex>[d]^{R''}\\
\mathcal{B}'\ar@<.5ex>[u]^{L'}\ar[r]_{G'}&{\mathcal{B}''}\ar@<.5ex>[u]^{L''}}
&  &
\end{array}%
\end{equation*}%
their horizontal composition is given by%
\begin{equation*}
\begin{array}{ccc}
\xymatrixrowsep{30pt}\xymatrixcolsep{5cm} \xymatrix{\ar@{} [dr]
|{\left(\begin{array}{cc} \zeta _{11}^{\prime }\ast \zeta _{11} & \zeta
_{12}^{\prime }\ast \zeta _{12} \\ \zeta _{21}^{\prime }\ast \zeta _{21} &
\zeta _{22}^{\prime }\ast \zeta _{22}\end{array}\right)}
\mathcal{A}\ar[r]^{F'F}\ar@<.5ex>[d]^R&{\mathcal{A}''}\ar@<.5ex>[d]^{R''}\\
\mathcal{B}\ar@<.5ex>[u]^{L}\ar[r]_{G'G}&{\mathcal{B}''}\ar@<.5ex>[u]^{L''}}
&  &
\end{array}
\qquad
\end{equation*}%
where%
\begin{eqnarray*}
\zeta _{11}^{\prime }\ast \zeta _{11} &=&F^{\prime }\zeta _{11}\circ \zeta
_{11}^{\prime }G, \\
\zeta _{12}^{\prime }\ast \zeta _{12} &=&\zeta _{12}^{\prime }\zeta
_{12}\circ L^{\prime \prime }G^{\prime }\eta ^{\prime }GR, \\
\zeta _{21}^{\prime }\ast \zeta _{21} &=&R^{\prime \prime }F^{\prime
}\epsilon ^{\prime }FL\circ \zeta _{21}^{\prime }\zeta _{21}, \\
\zeta _{22}^{\prime }\ast \zeta _{22} &=&\zeta _{22}^{\prime }F\circ
G^{\prime }\zeta _{22}.
\end{eqnarray*}
\end{remark}

The following result should be compared with \cite[Proposition I,6.9]%
{Gray-FormalCat}.

\begin{lemma}
\label{lem:equiv}Consider an adjoint square as in the following diagram.%
\begin{equation*}
\begin{array}{ccc}
\xymatrixrowsep{25pt}\xymatrixcolsep{3cm} \xymatrix{\ar@{} [dr]
|{\left(\begin{array}{cc} \zeta _{11} & \zeta _{12} \\ \zeta _{21} & \zeta
_{22}\end{array}\right)}
\mathcal{A}\ar[r]^F\ar@<.5ex>[d]^R&\mathcal{A}^{\prime
}\ar@<.5ex>[d]^{R^\prime}\\
\mathcal{B}\ar@<.5ex>[u]^{L}\ar[r]_G&\mathcal{B}^{\prime
}\ar@<.5ex>[u]^{L^\prime}} &  &
\end{array}%
\end{equation*}%
Assume that $F$ and $G$ are equivalences of categories. Then the following
assertions are equivalent.

\begin{itemize}
\item[(1)] $\zeta _{11}$ is an isomorphism.

\item[(2)] $\zeta _{22}$ is an isomorphism.
\end{itemize}
\end{lemma}

\begin{proof}
Let $F^{\prime }:\mathcal{A}^{\prime }\rightarrow \mathcal{A}$ be a functor
such that $\left( F^{\prime },F\right) $ is a category equivalence with
(invertible) unit $\eta ^{\left( F^{\prime },F\right) }$ and counit $%
\epsilon ^{\left( F^{\prime },F\right) }.$ Similarly we use the notation $%
\eta ^{\left( G^{\prime },G\right) }$ and $\epsilon ^{\left( G^{\prime
},G\right) }$. \newline
$\left( 2\right) \Rightarrow \left( 1\right) .$ Consider the following
adjoint squares, where the diagrams on the right-hand side are obtained by
rotating clockwise by $90$ degrees the ones on the left-hand side (the upper
index $c$ stands for ``clockwise'').
\begin{eqnarray*}
&&%
\begin{array}{ccc}
\xymatrixrowsep{25pt}\xymatrixcolsep{3cm} \xymatrix{\ar@{} [dr]
|{\left(\begin{array}{cc} \zeta _{11} & \zeta _{12} \\ \zeta _{21} & \zeta
_{22}\end{array}\right)} \mathcal{A}\ar[r]^F\ar@<.5ex>[d]^R&\mathcal{A}'
\ar@<.5ex>[d]^{R^\prime}\\
\mathcal{B}\ar@<.5ex>[u]^{L}\ar[r]_G&\mathcal{B}'\ar@<.5ex>[u]^{L^\prime}} &
&
\end{array}
\quad \overset{\circlearrowright }{\mapsto }\quad
\begin{array}{ccc}
\xymatrixrowsep{27pt}\xymatrixcolsep{4cm} \xymatrix{\ar@{} [dr]
|{\left(\begin{array}{cc} \zeta _{11}^{c} & \zeta _{12}^{c} \\ \zeta
_{21}^{c} & \zeta _{22}^{c}=\zeta_{11}\end{array}\right)}
\mathcal{B}\ar[r]^L\ar@<.5ex>[d]^G&\mathcal{A}\ar@<.5ex>[d]^{F}\\
{\mathcal{B}'}\ar@<.5ex>[u]^{G'}\ar[r]_{L'}&\mathcal{A}^{\prime
}\ar@<.5ex>[u]^{F'}} &  &
\end{array}
\\
&&%
\begin{array}{ccc}
\xymatrixrowsep{27pt}\xymatrixcolsep{4cm} \xymatrix{\ar@{} [dr]
|{\left(\begin{array}{cc} \zeta _{11}^{r} & \zeta _{12}^{r} \\ \zeta
_{21}^{r} & \zeta _{22}^{r}=(\zeta_{11})^{-1}\end{array}\right)}
\mathcal{A}\ar[r]^R\ar@<.5ex>[d]^F&\mathcal{B}\ar@<.5ex>[d]^{G}\\
{\mathcal{A}'}\ar@<.5ex>[u]^{F'}\ar[r]_{R'}&\mathcal{B}^{\prime
}\ar@<.5ex>[u]^{G'}} &  &
\end{array}
\quad \overset{\circlearrowright }{\mapsto }\quad
\begin{array}{ccc}
\xymatrixrowsep{27pt}\xymatrixcolsep{4cm} \xymatrix{\ar@{} [dr]
|{\left(\begin{array}{cc} \zeta _{11}^{rc} & \zeta _{12}^{rc} \\ \zeta
_{21}^{rc} & \zeta _{22}^{rc}=\zeta_{11}^{r}\end{array}\right)}
{\mathcal{A}'}\ar[r]^{F'}\ar@<.5ex>[d]^{R'}&\mathcal{A}\ar@<.5ex>[d]^{R}\\
{\mathcal{B}'}\ar@<.5ex>[u]^{L'}\ar[r]_{G'}&\mathcal{B}\ar@<.5ex>[u]^{L}} &
&
\end{array}%
\end{eqnarray*}
Now apply \cite[page 153]{Gray-FormalCat}, putting $\left( f,u\right) =\left(
G^{\prime },G\right) $, $\left( f^{\prime },u^{\prime }\right) =\left(
F^{\prime },F\right) $, $\left( g,v\right) =\left( L,R\right) $, $\left(
g^{\prime },v^{\prime }\right) =\left( L^{\prime },R^{\prime }\right)$, $%
\psi =\zeta _{22},$ $\widetilde{\psi }=\zeta _{11},$ $\widetilde{\widetilde{%
\psi }}=\zeta _{11}^{c},\theta =\zeta _{22}^{r},$ $\widetilde{\theta }=\zeta
_{11}^{r},$ $\widetilde{\widetilde{\theta }}=\zeta _{11}^{rc}.$ Then we
obtain that $\zeta _{11}^{rc}$ and $\zeta _{11}^{c}$ are mutual inverses.
Note that
\begin{equation*}
\zeta _{11}=\zeta _{22}^{c}\overset{(\ref{Adj4})}{=}FL\epsilon ^{\left(
G^{\prime },G\right) }\circ F\zeta _{11}^{c}G\circ \eta ^{\left( F^{\prime
},F\right) }L^{\prime }G
\end{equation*}%
so that, as a composition of isomorphisms, $\zeta _{11}$ is an isomorphism.
\newline
$\left( 1\right) \Rightarrow \left( 2\right) .$ This implication is shown in
a very similar fashion, by applying the dual result of \cite[page 153]%
{Gray-FormalCat}.
\end{proof}

\begin{definition}
Following \cite[1.4]{Berger-Mellies-Weber}, an adjoint square as in (\ref%
{diag:AdjSquare}) is called \textit{exact} whenever both $\zeta
_{11}:L^{\prime }G\rightarrow FL$ and $\zeta _{22}:GR\rightarrow R^{\prime
}F $ are isomorphisms. Note that this implies that the given diagram
commutes -up to isomorphism- when either the left adjoint functors or the
right adjoint functors are omitted.
\end{definition}

\begin{remark}
\label{rem:commdat}Consider a square of functors like in (\ref%
{diag:AdjSquare}) and assume that $GR=R^{\prime }F.$ Then we can set $\zeta
_{22}:={\mathsf{Id}}_{GR}$ and we get an adjoint square. This square is
exact if and only if $\left( F,G\right) :\left( L,R\right) \rightarrow
\left( L^{\prime },R^{\prime }\right) $ is a \emph{commutation datum} in the sense
of \cite[Definition 2.3]{AM-MonadSym}.
\end{remark}

\begin{proposition}
\label{pro:leftad}Consider two adjunctions $\left( L,R,\eta ,\epsilon
\right) $ and $\left( L',R',\eta ',\epsilon '\right) $ as in the diagram%
\begin{equation*}
\xymatrix{\mathcal{A}\ar[r]^{\id}\ar@<.5ex>[d]^R&\mathcal{A}\ar@<.5ex>[d]^{R^\prime}\\
\mathcal{B}\ar@<.5ex>[u]^{L}\ar[r]_G&\mathcal{B}^{\prime
}\ar@<.5ex>[u]^{L^\prime}}
\end{equation*}%
where $G$ is a functor such that $GR=R^{\prime }.$ Then the diagram above is
an adjoint square with matrix $\left( \zeta _{ij}\right) $, where $\zeta
_{22}:GR\rightarrow R^{\prime }$ is the identity.
 If $\eta $ is invertible, then $\left( RL^{\prime },G,\eta ^{\prime
},\eta ^{-1}\circ R\zeta _{11}\right) $ is an adjunction too.
\end{proposition}

\begin{proof}
Assume $\eta$ is invertible, set $F:=RL^{\prime }$ and%
\begin{eqnarray*}
\alpha &:=&\left[ FG=RL^{\prime }G\overset{R\zeta _{11}}{\longrightarrow }RL%
\overset{\eta ^{-1}}{\longrightarrow }\mathrm{Id}_{\mathcal{B}}\right] , \\
\beta &:=&\left[ \mathrm{Id}_{\mathcal{B}^{\prime }}\overset{\eta ^{\prime }}%
{\longrightarrow }R^{\prime }L^{\prime }=GRL^{\prime }=GF\right] .
\end{eqnarray*}%
We compute that
\begin{eqnarray*}
\alpha F\circ F\beta &=&\eta ^{-1}F\circ R\zeta _{11}F\circ F\eta ^{\prime
}=\eta ^{-1}RL^{\prime }\circ R\zeta _{11}RL^{\prime }\circ RL^{\prime }\eta
^{\prime } \\
&=&R\epsilon L^{\prime }\circ R\zeta _{11}RL^{\prime }\circ RL^{\prime }\eta
^{\prime }=R\left[ \left( \epsilon \circ \zeta _{11}R\right) L^{\prime
}\circ L^{\prime }\eta ^{\prime }\right] \\
&=&R\left( \zeta _{12}L^{\prime }\circ L^{\prime }\eta ^{\prime }\right)
=R\left( \left( \epsilon ^{\prime }\circ L^{\prime }\zeta _{22}\right)
L^{\prime }\circ L^{\prime }\eta ^{\prime }\right) =RL^{\prime }=F
\end{eqnarray*}%
and%
\begin{equation*}
G\alpha \circ \beta G=G\eta ^{-1}\circ GR\zeta _{11}\circ \eta ^{\prime
}G=G\eta ^{-1}\circ R^{\prime }\zeta _{11}\circ \eta ^{\prime }G=G\eta
^{-1}\circ \zeta _{21}=G\eta ^{-1}\circ \zeta _{22}L\circ G\eta =G,
\end{equation*}%
so $\left( RL^{\prime },G,\eta ^{\prime },\eta ^{-1}\circ R\zeta
_{11}\right) $ is an adjunction.
\end{proof}

\begin{remark}\label{rem-3.8}
Consider the following diagram
\begin{equation*}
\xymatrix{\mathsf{Bialg}(\mathfrak{M})\ar[r]^{\id}\ar@<.5ex>[d]^{\mathcal{P}}&\mathsf{Bialg}(\mathfrak{M})\ar@<.5ex>[d]^{P}\\
\mathsf{{Lie}_{p}}\ar@<.5ex>[u]^{\widetilde{\mathfrak{u}}}\ar[r]_{H_{\mathsf{{Lie}_{p}}}}&\mathfrak{M}\ar@<.5ex>[u]^{\widetilde{T}}}
\end{equation*}%
In the previous section we observed that $\widetilde{\eta }_{\mathrm{L}}:{\mathsf{id%
}}_{{\mathsf{{Lie}_{p}}}}\rightarrow \mathcal{P}\widetilde{\mathfrak{u}}$ is
an isomorphism. By Proposition \ref{pro:leftad}, the above is an adjoint square with matrix $\left( \zeta _{ij}\right) $, where $\zeta
_{22}:H_{\mathsf{{Lie}_{p}}}\mathcal{P}\rightarrow P$ is the identity.
Moreover, we also have that $\left( \mathcal{P}\widetilde{T},H_{\mathsf{{Lie}_{p}}},\widetilde{%
\eta },\widetilde{\eta }_{\mathrm{L}}^{-1}\circ \mathcal{P}\zeta
_{11}\right) $ is an adjunction. Thus $\mathcal{P}\widetilde{T}$ is a left adjoint of the functor $%
H_{\mathsf{{Lie}_{p}}}$. Let as write it explicitly on objects. For $V$ in $\mathfrak{M}$, we have
\begin{equation*}
\mathcal{P}\widetilde{T}V=\left( P\widetilde{T}V,\left[ -,-\right] ,-^{\left[
p\right] }\right)
\end{equation*}%
where $\left[ -,-\right] \ $and $-^{\left[ p\right] }$ are defined for every
$x,y\in P\widetilde{T}V$ by $\left[ x,y\right] =x\otimes y-y\otimes x\ $and $%
x^{\left[ p\right] }=x^{\otimes p}$ respectively.
\\The left adjoint of $H_{\mathsf{{Lie}_{p}}}$ appeared in the literature, for some particular base field $\Bbbk$ (e.g. $\Bbbk=\mathbb{Z}_2$), under the name of ``free restricted Lie algebra functor'' and several constructions of this functor can be found. We note that here, in our approach, no additional requirement on $\Bbbk$ is needed other than the finite characteristic. A similar description can be obtained even in characteristic zero.

\end{remark}

\begin{corollary}\label{coro:leftad}
  In the setting of Proposition \ref{pro:leftad}, assume that both $\eta $ and $\eta ^{\prime }$ are invertible and $G$ is
conservative. Then $\left( RL^{\prime },G,\eta ^{\prime },\eta ^{-1}\circ
R\zeta _{11}\right) $ is an adjoint equivalence. Moreover, $\zeta
_{11}:L^{\prime }G\rightarrow L$ is invertible (so that the adjoint square considered in Proposition \ref{pro:leftad} is a commutation datum).
\end{corollary}
\begin{proof}
  We give two alternative proofs of the first part of the statement. Then the last part follows by Lemma \ref{lem:equiv} as $\zeta _{22}$ is the identity.
    
  Proof I). By Proposition \ref{pro:leftad}, $\left( RL^{\prime },G,\eta ^{\prime
},\eta ^{-1}\circ R\zeta _{11}\right) $ is an adjunction too. Since $\eta ^{\prime
}$ is invertible, it remains to prove that $\eta ^{-1}\circ R\zeta _{11}$ is invertible i.e. that $R\zeta _{11}$ is.
 Since $R^{\prime }\zeta _{11}\circ \eta ^{\prime
}G=\zeta _{21}=\zeta _{22}L\circ G\eta $ is invertible, so is $R^{\prime
}\zeta _{11}$. Thus $GR\zeta _{11}=R^{\prime }\zeta _{11}$ is invertible.
Since $G$ is conservative, we conclude.

Proof II). Since $\eta $ and $\eta ^{\prime }$ are isomorphisms, we can apply
Lemma \ref{lem:Ber} to the following diagram%
\begin{equation*}
\xymatrixrowsep{15pt}\xymatrixcolsep{20pt}\xymatrix{&\mathcal{A}
\ar@<.5ex>[dl]^{R^{\prime}}\ar@<.5ex>[dr]^{R}\\
\mathcal{B^{\prime}}\ar@{.>}@<.5ex>[ur]^{L^{\prime}}\ar[dr]_{\mathrm{id}}&&%
\mathcal{B}\ar@{.>}@<.5ex>[ul]^{L}\ar@<.5ex>[dl]^{G}\\ &\mathcal{B^{\prime}}}
\end{equation*}%
to deduce that $\left( RL^{\prime },R^{\prime }L\right) $ is an adjoint
equivalence. By hypothesis, we have $R^{\prime }=GR.$ Thus we can use the
isomorphism $G\eta :G\rightarrow GRL=R^{\prime }L$ to replace $R^{\prime }L$
by $G$ in the adjunction. 
\end{proof}

\begin{remark}\label{rem-3.10}
We are now able to provide a different closing for the proof of Theorem \ref{teo:DecVec}. Once proved that $\widetilde{\eta }_{2}$ and $\widetilde{\eta }_{\mathrm{L}}$ are isomorphisms, we can apply Corollary \ref{coro:leftad} to the
following diagram%
\begin{equation*}
\xymatrixrowsep{25pt}\xymatrixcolsep{3cm}\xymatrix{\ar@{} [dr]
|{\left(\begin{array}{cc} \chi _{11} & \chi _{12} \\ \chi _{21} & \chi
_{22}\end{array}\right)}
\Bialg(\mathfrak{M})\ar[r]^{\id}%
\ar@<.5ex>[d]^{P_{2}}&\Bialg(\mathfrak{M})\ar@<.5ex>[d]^{\mathcal{P}}\\
{\mathfrak{M}}_{2}\ar@<.5ex>[u]^{{\widetilde{T}}_{2}}\ar[r]_{\Lambda}&\RLie%
\ar@<.5ex>[u]^{\widetilde{\mathfrak{u}}}}
\end{equation*}%
which is an adjoint square with matrix $\left( \chi _{ij}\right) $ where $\chi
_{22}:\Lambda P_{2}\rightarrow \mathcal{P}$ is the identity and $\chi _{11}:=\widetilde{\epsilon }_{\mathrm{L}}%
{\widetilde{T}}_{2}\circ \widetilde{\mathfrak{u}}\Lambda \widetilde{%
\eta }_{2}$. As a consequence $\left( P_{2}\widetilde{\mathfrak{u}},\Lambda ,\widetilde{\eta }_{%
\mathrm{L}},\widetilde{\eta }_{2}^{-1}\circ P_{2}\chi _{11}\right) $ is an
adjoint equivalence. Moreover  $\chi _{11}:\widetilde{\mathfrak{u}}\Lambda\to {\widetilde{T}}_{2}$ is invertible so that $\widetilde{\mathfrak{u}}\Lambda\cong {\widetilde{T}}_{2}$. Since we already know that  $\Lambda P_{2}=\mathcal{P}$, we can identify $(\widetilde{T}_{2},P_{2})$ with $(\widetilde{\mathfrak{u}},%
\mathcal{P})$ via $\Lambda $.
\end{remark}

\end{document}